\documentclass[11pt, reqno]{amsart}
\usepackage{amscd,amsfonts,amsmath,amssymb,amstext,amsthm}
\usepackage[all]{xy}
\usepackage{tikz-cd}
\usepackage{xcolor}
\usepackage{enumerate}
\usepackage{verbatim}
\usepackage{hyperref}
\usepackage{enumitem}
\usepackage{booktabs}
\usepackage{kotex}

\makeatletter

\theoremstyle{plain}
\newtheorem{theorem} {Theorem} [section]
\newtheorem{lemma} [theorem]{Lemma}
\newtheorem{proposition}[theorem]{Proposition}

\newtheorem{conjecture}[theorem]{Conjecture}
\theoremstyle{definition}

\newtheorem{example}[theorem]{Example}

\newtheorem{remark}[theorem]{Remark}
\numberwithin{equation}{section}

\title[Intermediate pseudoconvexity of fiber bundles]{Intermediate pseudoconvexity of fiber bundles}
\date{June 14, 2026}
\author{Masanori Adachi}
\address{Department of Mathematics, Faculty of Science, Shizuoka University, 836 Ohya, Suruga-ku, Shizuoka, 422-8529, Japan}
\email{adachi.masanori@shizuoka.ac.jp}
\author{Seungjae Lee}
\address{Kyungpook National University
80, Daehak-ro, Buk-gu, Daegu, 41566, Republic of Korea}
\email{seungjae@knu.ac.kr}
\author{Aeryeong Seo}
\address{Kyungpook National University
80, Daehak-ro, Buk-gu, Daegu, 41566, Republic of Korea}
\email{aeryeong.seo@knu.ac.kr}

\subjclass[2020]{Primary 32L05,
Secondary 32E10, 32F10, 32V15.}
\keywords{Holomorphic ball bundles,  
$q$-convexity, harmonic section, foliation}

\thanks{Authors equally contributed to this work. }

\begin{document}

\begin{abstract}
In this paper, we investigate the pseudoconvexity of locally trivial holomorphic ball bundles over compact Riemann surfaces of genus $\geq 2$, as well as the intermediate pseudoconvexity of their complements in the associated projective space bundles. 
Inspired by Brunella's work, we prove that any such ball bundle is $1$-convex, while its complement is $n$-convex, where $n$ denotes the dimension of the ball fiber, provided that the bundle admits a harmonic section with a point of maximal rank. 
\end{abstract}
\maketitle

\section{Introduction}

In the theory of several complex variables, fiber bundles often serve as important sources of examples and counterexamples. 
During the 1980s, locally trivial holomorphic disk bundles over compact complex manifolds were extensively studied.
Although such disk bundles are locally pseudoconvex, their global pseudoconvexity exhibits a variety of different phenomena.
Ohsawa \cite{Ohsawa0} constructed a Stein disk bundle over an elliptic curve, while Diederich and Forn{\ae}ss \cite{Diederich-Fornaess} exhibited a disk bundle over a Hopf surface that is not exhausted by pseudoconvex domains. 
In contrast, Diederich and Ohsawa \cite{Diederich-Ohsawa1} established that every disk bundle over a compact K\"ahler manifold is weakly 1-complete, i.e., it admits a smooth plurisubharmonic exhaustion function.
Subsequent work of Barrett \cite{Barrett} and of Diederich--Ohsawa \cite{Diederich-Ohsawa2} further investigated criteria for disk bundles over compact Riemann surfaces to be Stein or 1-convex.

More recently, the pseudoconvexity of other types of fiber bundles has also been investigated. 
In addition to carrying out a more detailed study of disk bundles, Deng and Forn{\ae}ss \cite{Deng-Fornaess} constructed locally trivial holomorphic ball bundles without nonconstant holomorphic functions over any given compact complex manifold with positive first Betti number. 
Seo \cite{Seo} established the weakly 1-completeness of locally trivial holomorphic fiber bundles over compact K\"ahler manifolds whose fibers  are bounded symmetric domains, under certain conditions on the holonomy representation.

The aim of this paper is to further clarify the global pseudoconvexity of holomorphic $\mathbb{B}^n$-bundles, where $\mathbb{B}^n $ denotes the unit ball in $\mathbb{C}^n$, $n \geq 1$. More precisely, we study conditions under which $\mathbb{B}^n$-bundles are 1-convex, a condition stronger than weakly 1-completeness, and investigate intermediate pseudoconvexity of their complements in the associated holomorphic $\mathbb{CP}^n$-bundles, where $\mathbb{CP}^n $ denotes the $n$-dimensional complex projective space.
While the complement of a disk bundle is again a disk bundle in the case $n=1$, the complement of a $\mathbb{B}^n$-bundle has pseudoconcave boundary for $n>1$. 
It is therefore natural to ask whether they still satisfy intermediate pseudoconvexity.
To investigate not only the 1-convexity of ball bundles but also the intermediate pseudoconvexity of their complements, we adopt an approach inspired by Brunella \cite{Brunella}. The key idea is to study the leafwise positivity of the normal bundle of the foliation naturally induced on the boundary of the $\mathbb{B}^n$-bundle.

To state our results, we now describe the setting of this paper briefly (see Section \ref{sect:positivity} for more details).
Let $\Sigma$ be a compact Riemann surface of genus $\geq 2$ and $E$ a locally trivial holomorphic $\mathbb{B}^n$-bundle over $\Sigma$.
Since $E$ is a flat bundle, $E$ is naturally embedded in the associated $\mathbb{CP}^n$-bundle $\widehat{E}$. 
We write the complement of the closure of $E$ in $\widehat{E}$ by $E' := \widehat{E} \setminus \overline{E}$, and the common boundary of $E$ and $E'$ by $\partial E$, which is oriented as the boundary of $E$.
The flat structure induces a horizontal foliation $\mathcal{F}$ on $\partial E$ by Riemann surfaces.
Write $T^{1,0}_{\partial E}$ for the CR holomorphic tangent bundle of $\partial E \subset \widehat{E}$. 
The CR normal line bundle $\mathcal N :=  (T^{1,0}_{\widehat{E}}|_{\partial E}) / T^{1,0}_{\partial E}$ can be seen as a leafwise holomorphic line bundle over $\partial E$. 

Our main result asserts the leafwise positivity of the CR normal line bundle $\mathcal N$ provided that the $\mathbb{B}^n$-bundle admits a harmonic section with a point of maximal rank.
\begin{theorem}\label{thm:main}
Let $\Sigma$ be a compact Riemann surface of genus $\geq 2$ and $E$ a holomorphic fiber bundle with fiber $\mathbb{B}^n$ embedded in the associated $\mathbb{CP}^n$-bundle $\widehat E$. Suppose that the $\mathbb{B}^n$-bundle $E$ over $\Sigma$ admits a harmonic section with a point of maximal rank.
Then the CR normal line bundle $\mathcal N$ over $\partial E$ admits a smooth hermitian metric of positive curvature along the horizontal foliation $\mathcal{F}$.
\end{theorem}

Note that when the genus of $\Sigma$ is zero or one, the assumption on the harmonic section in Theorem \ref{thm:main} is never fulfilled. Indeed, since the fundamental group $\pi_1(\Sigma)$ is trivial or abelian, the holonomy representation $\rho \colon \pi_1(\Sigma) \to \operatorname{Aut}(\mathbb{B}^n)$ has a fixed point in $\mathbb{B}^n$ or on the boundary $\partial\mathbb{B}^n$.
This gives rise to a locally constant section of the closure of the holomorphic $\mathbb{B}^n$-bundle, which contradicts the existence and uniqueness of the harmonic section with positive rank in the assumption.  
On the other hand, when the genus of $\Sigma$ is at least two, many such examples exist; see Examples \ref{ex:complex} and \ref{ex:totallyreal}. 

The following pseudoconvexity of $E$ and the intermediate pseudoconvexity of $E'$ are consequences of our main result through a Brunella-type construction.

\begin{theorem} \label{cor:1-n-covex}
Under the assumption of Theorem \ref{thm:main}, the $\mathbb{B}^n$-bundle $E$ is 1-convex and the $(\mathbb{CP}^n \setminus \overline{\mathbb{B}^n})$-bundle $E'$ is $n$-convex, but if $n \geq2$, then $E'$ is not $(n-1)$-convex. 
\end{theorem}

The disk bundle case $n = 1$ was established by Diederich and Ohsawa \cite{Diederich-Ohsawa1, Diederich-Ohsawa2} and Barrett \cite{Barrett} under a weaker assumption on the harmonic section. The main theorem in \cite{Seo} yields the weakly 1-completeness of $E$. Thus, our theorem on the $\mathbb{B}^n$-bundle $E$ may be regarded as a refinement of the result in \cite{Seo} in the case where the base manifold is one-dimensional.  

To explain our result on the  $(\mathbb{CP}^n \setminus \overline{\mathbb{B}^n})$-bundle $E'$, let us recall $q$-convexity and $q$-completeness, notions for intermediate pseudoconvexity that we shall use. 
{A complex manifold $\Omega$ is said to be \emph{$q$-convex} (resp., \emph{$q$-complete}) if $\Omega$ admits a smooth exhaustion function $\varphi$
such that $\sqrt{-1} \partial \overline \partial \varphi$ has at most $(q-1)$ non-positive eigenvalues outside a compact subset in $\Omega$ (resp., on all of $\Omega$).}
By a theorem of Greene and Wu \cite{Green-Wu}, every non-compact connected complex manifold $\Omega$ of dimension $N$ is $N$-complete (see also Ohsawa \cite{Ohsawa1}). Since $\dim_{\mathbb C} E' = n+1$, it follows that $E'$ is automatically $(n+1)$-complete. 
On the other hand, $E' \subset \widehat{E}$ is locally $n$-complete because $\partial E$ is foliated by Riemann surfaces (see, e.g., \cite[Proposition 3.5]{Ohsawa-Pawlaschyk}, where $n$-completeness corresponds to $(n-1)$-pseudoconvexity).
This leads naturally to the question of whether $E'$ is globally $n$-convex. Our theorem shows that this is indeed the case.

Our argument gives only partial result (Proposition \ref{prop:semi-positive}) for holomorphic ball bundles over compact K\"ahler manifolds of dimension $> 1$ since the hermitian metric of $\mathcal{N}$ that we construct is only leafwise semi-positive in general (Remark \ref{rem:non-positivity}).
In \cite[Theorems 6.1 and 6.4]{Seo}, it was shown that certain ball bundles over ball quotients can be hyperconvex, where the base dimension agrees with the fiber dimension.
On the other hand, Ohsawa's theorem \cite[Theorem 0.1]{Ohsawa2} yields that the ball bundle cannot be 1-convex if the fiber dimension $n$ is one and the base dimension $> 1$ since the complement of a real analytic Levi-flat hypersurface in compact K\"ahler manifold of dimension $\geq 3$ cannot be 1-convex.
In personal communication, Takakura communicated to the authors that ball bundles are not Stein if the fiber dimension is smaller than the base dimension since 
$N$-dimensional Stein manifold is homotopy equivalent to at most real $N$-dimensional CW complex. 
In view of these partial results, we would like to propose the following conjecture. 

\begin{conjecture}
Let $M$ be a compact K\"ahler manifold of dimension $k$ and $E$ a holomorphic $\mathbb{B}^n$-bundle over $M$ embedded in the associated $\mathbb{CP}^n$-bundle $\widehat{E}$. Write $E' := \widehat{E} \setminus \overline{E}$.
If $n < k$, $E$ and $E'$ are not 1-convex and $n$-convex, respectively.
\end{conjecture}

\section{Preliminaries}
\subsection{Automorphisms of the unit ball}
We shall make use of standard automorphisms of the unit ball $\mathbb{B}^n := \{ z  = (z_1, \dots, z_n)\in \mathbb{C}^n \mid |z| := \sqrt{\sum_{j=1}^n |z_j|^2} < 1 \}$.
We also write the unit disk as $\mathbb{D} := \mathbb{B}^1$.

For $z \in \mathbb{B}^n \setminus \{0\}$, let $[z]$ denote the one dimensional vector space spanned by $z$. Let $P_z$ denote the orthogonal projection from $\mathbb{C}^n$ onto $[z]$, and let $Q_z$ be the orthogonal projection  onto the orthogonal complement of $[z]$, satisfying $P_{z}+Q_{z} = \textup{Id}_{\mathbb C^n}$. 
For each $z\in \mathbb B^n$, let $T_z$ be the automorphism on $\mathbb{B}^n$ mapping $0$ to $z$,  given by
\begin{equation}\label{involution on the ball}
T_z (w) = \frac{z-P_z (w) - \sqrt{1-|z|^2} Q_{z} (w)}{1- w \cdot \bar z}, 
\qquad w \in \mathbb{B}^n,
\end{equation}
where $z \cdot w := \sum_{j=1}^n z_j w_j$ for $z, w \in \mathbb{C}^n$. Then $T_z$ is an involution (i.e., $T_z \circ T_z = \text{Id}_{\mathbb{B}^n}$)
and it satisfies the identity
\begin{equation}\label{calabi diastasis}
1-|T_z w |^2 = \frac{(1-|z|^2)(1-|w|^2)}{|1- z \cdot \overline w |^2}.
\end{equation}

\subsection{Harmonic section}
We shall recall some basic properties and examples of harmonic maps. For more details, we refer the readers to \cite{Eells-Lemaire1, Eells-Lemaire2}, for instance.

Let $(M, g)$ and $(N, h)$ be Riemannian manifolds, 
and let $f \colon M \to N$ be a smooth map. 
The energy functional $E(f)$ of the map $f$ is defined by
$$
E(f) := \frac{1}{2} \int_M \text{tr}_g(f^* h) \, dV_g
$$
where $dV_g$ denotes the volume form of $M$ with respect to $g$.
The map $f$ is said to be {\it harmonic} if it satisfies the Euler-Lagrange equation for the energy functional $E$, which requires that its tension field vanishes:
$$
\text{tr}_g(\nabla df) = 0
$$
where $\nabla$ is the connection on $T^*M \otimes f^*TN$ induced by the Levi-Civita connections of $M$ and $N$, 
considering $df$ as a section of this bundle.
Clearly, any totally geodesic map (i.e., where $\nabla df \equiv 0$) is necessarily harmonic. 
By the celebrated theorem of Eells-Sampson~\cite{Eells-Sampson}, 
if $M$ is compact and $N$ has non-positive sectional curvature, 
then any smooth map from $M$ to $N$ is homotopic to a smooth harmonic map.
By a theorem of Hartman \cite[(H)]{Hartman}, if $N$ has strictly negative sectional curvature, any harmonic map $f \colon M \to N$ is unique within its homotopy class, provided there exists a point in $M$ at which the rank of $f$ is greater than one.

Now, suppose $M$ is a K\"ahler manifold. A smooth map $f \colon M \to N$ is said to be {\it pluriharmonic} if for every complex curve $C \subset M$, the restriction $f|_C \colon C \to N$ is harmonic. Equivalently, this means
$$
(\nabla df)^{1,1} = 0,
$$
where $(\nabla df)^{1,1}$ denotes the $(1,1)$-part of $\nabla df$, viewed as an $f^*TN$-valued $2$-form on $M$. 
It is a well-known fact that any holomorphic or anti-holomorphic map between K\"ahler manifolds is pluriharmonic, and every pluriharmonic map is harmonic. In general, converse is not true, however if $M$ is a compact K\"ahler manifold and $N$ is a Riemannian manifold of non-positive Hermitian sectional curvature, then any harmonic map from $M$ to $N$ is pluriharmonic \cite{Siu, Sampson}.
Moreover, if $N$ is also a compact K\"ahler manifold with strongly negative curvature, and if there exists a point in $M$ at which the real rank of $f$ is strictly greater than $2$, then any such harmonic map must be either holomorphic or anti-holomorphic \cite{Siu}.

Let $M$ be a compact Riemannian manifold and $G$ be a semisimple algebraic group. By a foundational theorem of Corlette \cite{Corlette}, a representation $\rho \colon \pi_1(M) \to G$ is reductive if and only if there exists a $\rho$-equivariant harmonic map from the universal cover to the symmetric space $G/K$, where $K$ is a maximal compact subgroup of $G$. Equivalently, this means that the representation $\rho$ is reductive if and only if the associated fiber bundle $M \times_\rho (G/K)$ admits a harmonic section.

{

In this paper, we investigate holomorphic $\mathbb{B}^n$-bundles over a compact K\"ahler manifold $M$ that admit a harmonic section. It is worth mentioning that such a bundle admits at most one harmonic section, as $\mathbb B^n$ is contractible.
When the base space $M$ is a compact K\"ahler manifold of complex dimension at least $2$, such a harmonic section is often forced to be holomorphic or anti-holomorphic due to rigidity phenomena. For instance, if $f \colon M \to N$ is a harmonic map into a quotient $N$ of $\mathbb{B}^n$, and the real rank of $f$ is greater than $2$ at some point in $M$, then $f$ is either holomorphic or anti-holomorphic \cite[Corollary 3.7]{Carlson-Toledo}. See also \cite{Corlette}.

On the other hand, When $\Sigma$ is a compact Riemann surface, there are plenty of examples where the bundle only admits harmonic section which is neither holomorphic nor anti-holomorphic. We shall give below basic examples based on \cite{Goldman-Kapovich-Leeb}. 
 }
\begin{example} \label{ex:complex}
Let $\Sigma$ be a compact Riemann surface of genus $g \geq 2$ and write its fundamental group $\pi_1(\Sigma)$ as $\Gamma$. Take a group homomorphism $\phi \colon \Gamma \to \operatorname{Aut}(\mathbb{D}) = PSU(1,1)$ from the  Teichm\"uller space of $\Sigma$, namely, we assume that $\mathbb{D}/\phi(\Gamma)$ is a compact Riemann surface of the same genus $g$.
We may choose such a $\phi$ so that $\mathbb{D}/\phi(\Gamma)$ is neither biholomorphic nor anti-biholomorphic to $\Sigma$. 
Then the harmonic diffeomorphism from $\Sigma$ to $\mathbb{D}/\phi(\Gamma)$ (see, e.g., \cite{Jost}) induces a harmonic section $s$ of the associated $\mathbb{D}$-bundle $\Sigma \times_\phi \mathbb{D}$ that is neither holomorphic nor anti-holomorphic.
It is a classical fact that there exists a lift of $\phi$ to $\widetilde{\phi} \colon \Gamma \to SU(1,1)$ (cf. \cite[Proposition 2.6]{Goldman-Kapovich-Leeb}). 
By composing ${\phi}$ with the natural homomorphism $ \psi \colon SU(1,1)  \hookrightarrow SU(n,1) \rightarrow PSU(n,1)$, we obtain a group homomorphism $\rho = \psi \circ \widetilde \phi \colon \Gamma \to \operatorname{Aut}(\mathbb{B}^n) =PSU(n,1)$. 
Since the holomorphic embedding $\mathbb{D} \to \mathbb{B}^n$ is $\psi$-equivariant, 
it induces a holomorphic embedding $\iota \colon \Sigma \times_\phi \mathbb{D} \to \Sigma \times_\rho \mathbb{B}^n$ of holomorphic fiber bundles. 
So, $\widetilde s = \iota \circ s$ is the unique harmonic section of $\Sigma \times_\rho \mathbb{B}^n$ which is neither holomorphic nor anti-holomorphic.
\end{example}

\begin{example} \label{ex:totallyreal}
Let $\Sigma$, $\Gamma$, $\phi$ and $s$ as in the previous example \ref{ex:complex}. 
Regard $\phi \colon \Gamma \to \operatorname{Aut}(\mathbb{D}) = SO(2,1)^o$, where $SO(2,1)^o$ denotes the identity component of $SO(2,1)$.
By composing ${\phi}$ with the natural homomorphism $ \psi \colon SO(2,1) \hookrightarrow SU(n,1) \rightarrow PSU(n,1)$, $n \geq 2$, we obtain a group homomorphism $\rho = \psi \circ\phi \colon \Gamma \to \operatorname{Aut}(\mathbb{B}^n) =PSU(n,1) \cong PU(n,1)$. 
Since the totally real totally geodesic embedding $\mathbb{D} \to \mathbb{B}^n$, $z \mapsto (\operatorname{Re} z, \operatorname{Im} z, 0, \dots, 0)$, is $\psi$-equivariant, it induces a smooth embedding $\iota \colon \Sigma \times_\phi \mathbb{D} \to \Sigma \times_\rho \mathbb{B}^n$ of smooth fiber bundles. 
We see  that $\widetilde s = \iota \circ s$ is the unique harmonic section of $\Sigma \times_\rho \mathbb{B}^n$ which is neither holomorphic nor anti-holomorphic.
\end{example}

\subsection{CR structures and D'Angelo forms} 
We shall explain our notation for CR geometry, in particular, D'Angelo $(1,0)$-form, and prove an intermediate pseudoconvexity result (Proposition \ref{prop:q-convexity}). Our formulation is based on \cite{Adachi-Yum} and \cite{Dallara-Mongodi}, primarily the latter. Some more details can be found in the paper of Dall'Ara and Mongodi \cite{Dallara-Mongodi}.

A \emph{CR manifold of hypersurface type} is an orientable smooth real manifold $M$ of dimension $2n+1$ equipped with a rank $n$ complex subbundle $T^{1,0}_M \subset T_{\mathbb C}(M)$ of the complexification of the real tangent bundle over $M$ such that $T^{1,0}_M \cap \overline{T^{1,0}_M}$ is of rank zero and the integrability condition, $[\Gamma(T^{1,0}_M), \Gamma(T^{1,0}_M) ] \subset \Gamma(T^{1,0}_M)$, is fulfilled. 
We write $H(M) :=\operatorname{Re}(T^{1,0}_M)$.
A \emph{pseudo-hermitian structure} on $M$ is a non-vanishing smooth real 1-form $\theta$ on $M$ that annihilates $H(M)$. A real smooth vector field $T$ on $M$ is said to be \emph{$\theta$-normalized} if $\theta(T) = 1$. 
Such a vector field gives smooth direct sum decompositions
$$T_\mathbb{C}(M) = T^{1,0}_M \oplus \overline{T^{1,0}_M} \oplus \mathbb{C} T, \quad T(M) = H(M) \oplus \mathbb{R} T.$$
Note that in some papers including \cite{Adachi-Yum} use pure-imaginary 1-forms and vector fields instead of real ones.

\begin{example} 
Let $\Omega$ be a relatively compact domain in a complex manifold $\widehat{\Omega}$ with smooth boundary $M$. The smooth real hypersurface $M$ is canonically endowed with a CR structure $T^{1,0}_M := T^{1,0}_{\widehat{\Omega}} \cap T_\mathbb{C}(M)$ of hypersurface type.  
A smooth defining function $r_0$ of $\Omega$, i.e., a smooth function $r_0 \colon \widehat{\Omega} \to \mathbb{R}$ such that $\Omega = r_0^{-1}((-\infty, 0))$ and $dr_0 \neq 0$ on $M$, naturally induces a pseudo-hermitian structure $\theta_{r_0} = \sqrt{-1} (\partial - \overline{\partial})r_0$ on $M$. 
This $\theta_{r_0}$ is \emph{compatible with the orientation} of the boundary $M$ in the sense that, for any $\theta_{r_0}$-normalized vector field $T$, $J_{\widehat\Omega}(T)$ is outward, $dr_0(J_{\widehat\Omega}(T)) = \theta(T) = 1$, where $J_{\widehat\Omega}$ denotes the complex structure of $\widehat{\Omega}$.
Any other pseudo-hermitian structure $\theta$ on $M$ can be written as $\theta = f \theta_{r_0}$ for some non-vanishing real smooth function $f$ on $M$. If $f > 0$, then $\theta$ is also compatible with the orientation of $M$ and induced from another defining function $r$ of $\Omega$, so that $\theta = \theta_r$ (cf. \cite[Lemma 3.1] {Adachi-Yum} or \cite[Proposition 3.7]{Dallara-Mongodi}). 
\end{example}

For a CR manifold $M$ with pseudo-hermitian structure $\theta$, its \emph{Levi form} $\lambda_p$ at $p \in M$ is the hermitian form on $T^{1,0}_{M,p}$ defined by
\[
\lambda_p(X_p, Y_p) := \frac{\sqrt{-1}}{4}\theta_p([X, \overline{Y}]_p) \quad \text{for $X_p, Y_p \in T^{1,0}_{M,p}$},
\]
where $X$ and $Y$ denote local smooth sections of $T^{1,0}_M$ extending $X_p$ and $Y_p$ respectively. 
We say that $(M, \theta)$ is \emph{pseudoconvex} (resp., \emph{pseudoconcave}) when its Levi form is semi-positive (resp., semi-negative) everywhere on $M$. 
When $M$ is a boundary of a domain, we simply say that $M$ is pseudoconvex or pseudoconcave since the semi-positivity or semi-negativity of the Levi form is independent of the choice of $\theta$ which is compatible with the orientation of $M$. 

The \emph{D'Angelo $(1,0)$-form} $\alpha_{1,0}$ is a smooth differential form of type $(1,0)$ defined by
$$
\alpha_{1,0}(X_p) := \theta_p([T,X]_p)\quad \text{for $X_p \in T^{1,0}_{M,p}$},
$$
where $T$ is a $\theta$-normalized vector field defined in a neighborhood of $p$.
Note that, on the kernel of the Levi form $\operatorname{Ker} \lambda_p \subset T^{1,0}_{M,p}$, the D'Angelo $(1,0)$-form $\alpha_{1,0}$ is independent of the choice of the local $\theta$-normalized vector field $T$ (cf. \cite[Lemma 2.5]{Adachi-Yum} or \cite[Proposition-Definition 4.1]{Dallara-Mongodi}).

For domains in complex manifolds with smooth pseudoconvex boundary, based on the breakthrough work of Liu \cite{Liu}, Yum \cite{Yum} found a formula to express the Diederich--Forn{\ae}ss index of the domain in terms of the D'Angelo $(1,0)$-form $\alpha_{1,0}$ on the boundary. 
The formula involves a hermitian form $\overline{\partial} \alpha_{1,0}$ on the kernel of the Levi form $\operatorname{Ker} \lambda$, where $\overline{\partial}$ denotes the tangential Cauchy--Riemann operator. 
The fact that $\overline{\partial}\alpha_{1,0}$ induces a hermtian form follows from the pseudoconvexity of $M$ based on an observation by Boas and Straube in \cite{Boas-Straube} (cf. \cite[Lemma 2.6]{Adachi-Yum} or \cite[Proposition-Definition 4.1, Proposition 4.2]{Dallara-Mongodi}).

The following proposition is implicitly contained in the formula for the Diederich--Forn{\ae}ss index.

\begin{proposition}[cf. {\cite[Corollary 5.3 (1)]{Adachi-Yum}}] \label{prop:adachi-yum}
Let $\Omega$ be a relatively compact domain in a complex manifold with smooth pseudoconvex boundary $M$. 
Assume that $M$ admits a pseudo-hermitian structure $\theta$ such that $\overline{\partial}\alpha_{1,0} > 0$ holds on $\operatorname{Ker} \lambda$. Then the Diedeirch--Forn{\ae}ss index of $\Omega$ is positive, namely, there exists a smooth defining function $r$ of $\Omega$ and $\eta \in (0,1)$ such that $-(-r)^\eta$ is strictly plurisubharmonic on $\Omega \setminus K$ for some compact subset $K \subset \Omega$. In particular, $-\log(-r)$ is also strictly plurisubharmonic on $\Omega \setminus K$ and $\Omega$ is 1-convex. 
\end{proposition} 

In a similar vein, the following intermediate pseudoconvexity holds for domains with \emph{pseudoconcave} boundary.

\begin{proposition}
\label{prop:q-convexity}
Let $\Omega$ be a relatively compact domain in a complex manifold $\widehat{\Omega}$ with smooth pseudoconcave boundary $M$. 
Assume that $M$ admits a pseudo-hermitian structure $\theta$ and a smooth subbundle $\mathcal{K} \subset T^{1,0}_{M}$ of positive rank $\kappa$ such that $\lambda = 0$ and $\overline{\partial}\alpha_{1,0} > 0$ hold on $\mathcal{K}$. Then $\Omega$ is $(\dim_\mathbb{C} \Omega - \kappa)$-convex. More precisely, there exists a smooth defining function $r$ of $\Omega$ such that the Levi form of $-\log (-r)$ has at least $(\kappa +1)$ positive eigenvalues on $\Omega \setminus K$ for some compact subset $K \subset \Omega$. 
\end{proposition}

\begin{proof}
We may suppose that $\theta$ is compatible with the orientation of $M$ by replacing $\theta$ with $-\theta$ if necessary. Let $r$ be a defining function of $\Omega$ such that $\theta = \theta_r$. 
%For simplicity, we let $\delta := -r$. 
We shall show that $-\log (-r)$ has the desired property by adapting the argument in \cite[\S6.2]{Dallara-Mongodi}, which we follow almost line by line.

Take a boundary point $p_0 \in M$ and a smooth $(1,0)$-vector field $\tau$ near $p_0$ in $\widehat{\Omega}$ such that $dr(\tau) = 1$.
Also take a local smooth frame of $\mathcal{K}$ near $p_0$ in $M$ and extend it to linearly independent smooth $(1,0)$-vector fields $\{\xi_1, \cdots, \xi_\kappa \}$ near $p_0$ in $\widehat{\Omega}$ so that $dr(\xi_i) = 0$ for any $j = 1,\dots,\kappa$. We write $\xi_{\kappa+1} := \tau$.
In view of the min-max principle, it is enough to find a neighborhood $V$ of $p_0$ in $\widehat{\Omega}$ such that $\sqrt{-1}\partial\overline{\partial}(-\log (-r))$ is positive-definite on the subspace $\{(\xi_1)_p, \cdots, (\xi_\kappa)_p, (\xi_{\kappa+1})_p = \tau_p \}$ of $T^{1,0}_{\Omega,p}$ at every $p \in V \cap \Omega$.
Since
$$
(-r)\partial\overline{\partial}(-\log (-r))
= \partial\overline{\partial}r + 
\frac{\partial r \wedge \overline{\partial} r}{(-r)},
$$
it holds that
$$
\biggl[ (-r) \partial\overline\partial (-\log(-r))((\xi_j)_p, (\overline{\xi}_k)_p)\biggr] = \biggl[  \partial\overline\partial r((\xi_j)_p, (\overline{\xi}_k)_p) + \frac{1}{2}\frac{\delta_{j,\kappa+1}\delta_{k,\kappa+1}}{(-r)}\biggr],
$$
where $\delta_{j,k}$ denotes the Kronecker delta. 
Notice that $\partial\overline\partial r(\tau_p, \overline{\tau}_p) + 1/(-2r) > 0$ on $V \cap \Omega$ if we take $V$ enough small. 
Thus, by Sylvester's criterion, this matrix is positive-definite if and only if
$$
(-r)(\partial\overline{\partial} r (\xi_p,\overline{\xi}_p)\partial\overline{\partial} r (\tau_p,\overline{\tau}_p) - |\partial\overline{\partial} r (\xi_p,\overline{\tau}_p)|^2) + \frac{1}{2}\partial\overline{\partial} r (\xi_p,\overline{\xi}_p)> 0
$$
for any non-zero $\xi_p \in \operatorname{span}\{(\xi_1)_p, \cdots, (\xi_\kappa)_p \}$.
Writing $\xi_p = \sum_{j=1}^\kappa a_j \cdot (\xi_j)_p$ by $a \in \mathbb{C}^{\kappa}$, we regard this expression as a quadratic form $Q(p; a)$ on $\mathbb{C}^\kappa$ depending smoothly on the point $p$ in a neighborhood of $p_0$ in $\widehat{\Omega}$.
Since $\lambda(\xi_p, \overline{\xi}_p) = 0$ when $p \in M$, exactly the same computation as in \cite[\S6.1 \& \S6.2]{Dallara-Mongodi} yields
$$
Q(p;a) = 0, \quad d_pQ(p;a)(\operatorname{Re} \tau) = -\overline{\partial}\alpha_{1,0}(\xi_p,\overline{\xi}_p) < 0
$$
when $p \in M$. 
From \cite[Lemma 6.9]{Dallara-Mongodi}, with orientation taken into account, we find a neighborhood $V$ of $p_0$ in $\widehat{\Omega}$ such that $Q(p;a) > 0$ for every $p \in V \cap \Omega$ and non-zero $a \in \mathbb{C}^\kappa$. 
This completes the proof. 
\end{proof}

\section{Semi-positivity of the CR normal line bundle}\label{sect:positivity}

In this section, we study the CR normal line bundles of flat sphere bundles over K\"ahler manifolds and prove their leafwise semi-positivity under the assumption that the associated ball bundle admits a harmonic section (Proposition \ref{prop:semi-positive}).

Let $\Sigma$ be a K\"ahler manifold of arbitrary dimension and $\pi_E \colon E \to \Sigma$ a holomorphic fiber bundle over $\Sigma$ with fiber the unit ball $\mathbb{B}^n$. Since $\mathbb{B}^n$ is Kobayashi hyperbolic, the structure group reduces to $\mathrm{Aut}(\mathbb{B}^n)$, and any such bundle arises from a representation of $\pi_1(\Sigma)$. Consequently, there exists a homomorphism
$$
\rho \colon \pi_1(\Sigma) \to \mathrm{Aut}(\mathbb{B}^n)
$$
such that $E$ can be realized as the quotient
$$
E \cong (\widetilde{\Sigma} \times \mathbb{B}^n)/\sim,
$$
where $\widetilde{\Sigma}$ denotes the universal cover of $\Sigma$, and the equivalence relation is given by
$$
(\zeta,w) \sim (\zeta',w')
\quad\text{if and only if}\quad
\exists\, \gamma \in \pi_1(\Sigma)
\ \text{such that}\
\zeta' = \gamma \cdot \zeta,\quad
w' = \rho(\gamma)\cdot w.
$$
Let 
$
\widehat E := (\widetilde \Sigma \times \mathbb{CP}^n)/\sim
$
be the total space of the associated $\mathbb{CP}^n$-bundle $\pi \colon \widehat E \to \Sigma$ obtained from the same representation, which contains $E$ as a locally pseudoconvex domain.

Let $\partial E = (\widetilde \Sigma \times \partial\mathbb B^n)/\sim$ be the boundary of $E$ in $\widehat E$, a sphere bundle $\pi_{\partial E} \colon \partial E \to \Sigma$ over $\Sigma$.
Consider on $\widetilde \Sigma\times \partial\mathbb B^n$ the product foliation whose leaves are the slices
$$
L_{w}:=\widetilde \Sigma\times\{w\},\qquad w \in\partial\mathbb B^n.
$$
The $\pi_1(\Sigma)$ action $
g\cdot(\zeta,w):=(\gamma\cdot\zeta,\ \rho(\gamma)\cdot w)
$
preserves this foliation, hence it descends to a well-defined foliation $\mathcal F$ on $\partial E$. 

Write $T_{\mathbb{C}} (\partial E)$ for the complexification of the real tangent bundle over $\partial E$. The CR holomorphic tangent bundle $T^{1,0}_{\partial E}$ is defined by 
\[
T^{1,0}_{\partial E} :=  T^{1,0}_{\widehat{E}} \cap T_{\mathbb{C}} (\partial E)
\]
where $T^{1,0}_{\widehat{E}}$ denotes the holomorphic tangent bundle of the ambient space $\widehat{E}$. 
The holomorphic tangent bundle $T^{1,0}_\mathcal{F}$ of the leaves of the foliation $\mathcal F$ gives a CR subbundle of $T^{1,0}_{\partial E}$, which agrees with the kernel of the Levi form.
The intersection with the vertical holomorphic tangent bundle $\ker d\pi \subset T^{1,0}_{\widehat E}$ gives another CR subbundle
$$
\mathcal E := \ker d\pi \cap T^{1,0}_{\partial E},
$$
and $T^{1,0}_{\partial E}$ splits into these two CR subbundles:
$$
T^{1,0}_{\partial E} = \mathcal{E} \oplus T^{1,0}_{\mathcal F}.
$$
Since $\mathcal E$ agrees with the CR holomorphic tangent bundle of the boundary sphere in each fiber, the Levi form is positive definite on $\mathcal E$. 
The real hypersurface $\partial E$ is therefore pseudoconvex of constant Levi rank $(n-1)$. 

Now we consider normal bundles. First we consider the holomorphic normal bundle of the foliation $\mathcal{F}$ given by
$$
 N^{1,0}_{\mathcal F} := (T^{1,0}_{\widehat E}|_{\partial E})/T^{1,0}_\mathcal F \simeq (\ker d\pi) |_{\partial E},
$$
which is CR isomorphic to the restriction of the vertical  holomorphic tangent bundle.
Also we consider the CR normal line bundle 
$$
\mathcal{N}:=(T^{1,0}_{\widehat{E}}|_{\partial E}) / T^{1,0}_{\partial E} \simeq N^{1,0}_{\mathcal F}/\mathcal E.
$$
Note that these CR vector bundles $N^{1,0}_{\mathcal F}$ and $\mathcal{N}$ can be seen as leafwise holomorphic bundles along $\mathcal{F}$ since CR functions are holomorphic along the leaves of $\mathcal{F}$.

We shall construct hermitian metrics on $N^{1,0}_{\mathcal F}$ and $\mathcal{N}$ from a given smooth section $s \colon \Sigma \rightarrow E$ of the $\mathbb{B}^n$-bundle $E$. Let $\widetilde D \subset \widetilde \Sigma$ be the interior of a fundamental domain for the action of $\pi_1(\Sigma)$ and $D \subset \Sigma$ is the open image of $\widetilde D$ by the canonical projection $\widetilde \Sigma \rightarrow \Sigma$. 
Then, using the canonical projection $\widetilde{\Sigma} \times \mathbb{B}^n \to E$, we obtain a biholomorphism $E |_{\pi^{-1}_{E} (D)} \cong D\times \mathbb{B}^n $.
 We regard $s$ as $\mathbb{B}^n$-valued smooth function over $D$ using this coordinate.
For each $\zeta \in D$, we note that $T_{s(\zeta)} : \mathbb{B}^n \rightarrow \mathbb{B}^n$ is extended biholomorphically across over the boundary of $\mathbb{B}^n$, i.e., there exists two open sets $\overline{\mathbb{B}^n} \subset V_j^{\zeta} \subset \mathbb{C}^n$, $j = 1, 2$, and $T_{s (\zeta)} : V_1^\zeta \rightarrow V_2^\zeta$ is a biholomorphism.
We define a hermitian metric on $N^{1,0}_{\mathcal{F}} |_{\pi_{\partial E}^{-1} (D)}$ by identifying $N^{1,0}_{\mathcal{F}} \big|_{\pi_{\partial E}^{-1}(\zeta)}  \cong T^{1,0}_{V_1^{\zeta}}\big|_{\partial \mathbb{B}^n} $ and using the hermitian metric of $V_1^\zeta$ given by
\begin{equation}\label{fiberwise metric}
g_{\zeta} := T_{s(\zeta)}^* \big(g_{\mathbb{C}^n} \big)
\end{equation}
{where 
$$
g_{\mathbb{C}^n} = \sqrt{-1} \sum_{\ell=1}^{n} dt_\ell \wedge d \bar t_{\ell}
$$
and $(t_1, \cdots, t_n)$ is the standard coordinate system on $V_2^\zeta$. }

\begin{lemma}\label{lem:metric}
For any $\gamma \in \pi_1 (\Sigma)$,
$$
g_{\zeta} = \rho(\gamma)^* g_{\gamma \zeta}
$$
holds, and $g$ defines a global hermitian metric on $N_{\mathcal{F}}^{1,0}$.
\end{lemma}
\begin{proof}
We need to show that the construction of $g_\zeta$ does not depend on the choice of a fundamental domain $\widetilde D \subset \widetilde \Sigma$ for the action of $\pi_1(\Sigma)$.
%with respect to the canonical projection $\widetilde \Sigma \rightarrow \Sigma$. 
To see this, we observe that, for $\gamma \in \pi_1(\Sigma)$,
\begin{equation}\label{unitary}
U:=T_{s(\gamma \zeta)} \circ \rho (\gamma) \circ T_{s(\zeta)}^{-1} \in U(n)
\end{equation}
since 
\begin{equation*}
\begin{aligned}
\big( T_{s(\gamma\zeta)} \circ \rho(\gamma) \circ T_{s(\zeta)}^{-1} \big)(0)&=T_{s(\gamma \zeta)} (  \rho(\gamma) ( s(\zeta) ) )
=T_{s(\gamma \zeta)} (s(\gamma \zeta)) =0.
\end{aligned}
\end{equation*}
Here, we use that $s$ is a $\rho$-equivariant map. 
Then by \eqref{unitary}, we have
\begin{equation}\nonumber
\begin{aligned}
(T_{s(\zeta)}^* \circ U^*)(g_{\mathbb{C}^n}) =  \rho(\gamma)^* \circ T_{s(\gamma \zeta)}^* (g_{\mathbb{C}^n}). 
\end{aligned}
\end{equation}
Since $U$ is unitary, it follows that
$
g_{\zeta} = \rho(\gamma)^* g_{\gamma\zeta}.
$
Therefore, the proof is completed.
\end{proof}

\begin{proposition} \label{prop:semi-positive}
Let $\Sigma$ be a K\"ahler manifold and $E$ a holomorphic fiber bundle over $\Sigma$ with fiber $\mathbb{B}^n$ embedded in the associated $\mathbb{CP}^n$-bundle $\pi \colon \widehat{E} \to \Sigma$.
Assume that a harmonic section $s$ of $E$ is given. Then the CR normal line bundle $\mathcal{N}$ over $\partial E$ admits a smooth hermitian metric of semi-positive curvature along the horizontal foliation $\mathcal{F}$. Moreover, the degeneracy locus of the curvature is contained in $\pi_{\partial E}^{-1}(C)$ where 
$C := \{ [\zeta] \in \Sigma \mid \text{$ds_\zeta: T_\zeta\widetilde{\Sigma} \to T_{s(\zeta)}\mathbb{B}^n$ is not injective} \}$.
\end{proposition}
\begin{proof}
Lemma \ref{lem:metric} yields a hermitian metric $g$ on $N^{1,0}_{\mathcal F}$. 
Take the orthogonal decomposition $N^{1,0}_{\mathcal F} = \mathcal{E} \oplus \mathcal{E}^\perp$ with respect to $g$ and induce a metric $g_{\mathcal{N}}$ on $\mathcal{N}$ by identifying $\mathcal{N}$ with $\mathcal{E}^\perp$ smoothly.
We show that the metric $g_{\mathcal N}$ has semi-positive curvature along $\mathcal F$. 

First we compute the local expression of $g_\mathcal{N}$. We trivialize $N_\mathcal{F}^{1,0} |_{\pi_{\partial E}^{-1} (D)} \cong (D_{\zeta} \times \partial \mathbb{B}^n_w )\times \mathbb{C}^n$, by using the isomorphism 
$
N^{1,0}_{\mathcal{F},p} \cong T^{1,0}_{V_1^\zeta, w} \cong \mathbb{C}^n
$
for each point $p = [(\zeta, w ) ] \in \partial E$, where $D \subset \Sigma$ is the open image of the interior of a fundamental domain $\widetilde D \subset \widetilde \Sigma$ for the canonical projection $\widetilde \Sigma \rightarrow \Sigma$.
Let $\tau := \sum_{j=1}^{n} w_j \frac{\partial}{\partial w_j}$, a leafwise holomorphic section of $N^{1,0}_{\mathcal{F}}$ over ${\pi^{-1}_{\partial E} (D)}$. This $\tau$ induces a leafwise holomorphic frame $[\tau]$ of $\mathcal{N}$ over $\pi^{-1}_{\partial E} (D)$.

On the other hand, a {global} smooth normal frame of ${\mathcal E}^\perp$ can be chosen as $\nu := \sum_{\ell=1}^n t_\ell \frac{\partial}{\partial t_\ell}$ where $(t_1, \dots, t_n)$ denotes the coordinate of $V_2^\zeta$. 
{This local expression of $\nu$ gives a global vector field since the coordinates change from $V_2^\zeta$ to $V_2^{s(\zeta)}$ is given by 
$$
T_{s(\gamma \zeta)} \circ \rho( \gamma) \circ T_{s(\zeta)}^{-1}= U \in U(n)
$$
by equation \eqref{unitary} and the Euler vector field is preserved by $U(n)$.}

From the definition of $g_{\mathcal{N}}$,
$$
|\tau|^2_{g_{\mathcal{N}}} := g_{\mathcal{N}}([\tau],[\tau]) = g(g(\tau, \nu)\nu, g(\tau, \nu)\nu) = |g(\tau, \nu) |^2.
$$
Since
\begin{equation*}
\begin{aligned}
g(\tau, \nu) &= \left(T_{s(\zeta)}^* ( g_{\mathbb{C}^n} )\right)(\tau,\nu) %= g_{\mathbb{C}^n} (T_{*} \gamma, \nu) 
=
\sum_{j,\ell}  w_j \frac{\partial (T_{s(\zeta)} w)_\ell}{\partial w_j} \overline{(T_{s(\zeta)}w)_\ell} 
= \sum_{j}  w_j \frac{\partial |T_{s(\zeta)} w|^2}{\partial w_j},
\end{aligned}
\end{equation*}
by \eqref{calabi diastasis} we have
\begin{equation*}
\begin{aligned}
g(\tau, \nu) &= \sum_j w_j \frac{\partial}{\partial w_j} \frac{(1-|s(\zeta)|^2)(1-|w|^2)}{|1-s(\zeta)\cdot\overline{w}|^2} = \frac{1-|s(\zeta)|^2}{|1-s(\zeta)\cdot \overline{w}|^2} 
\end{aligned}
\end{equation*}
where we used $|w|= 1$. This gives the local expression of $g_\mathcal{N}$. 

Next we compute the leafwise curvature of $g_\mathcal{N}$. Fix a point $[(\zeta_0,w_0)] \in \partial E$. By changing local trivialization of $E$, we may assume that $s(\zeta_0) = 0 \in \mathbb{B}^n$.
Since $s$ is pluriharmonic, $s_{\zeta_j \overline{\zeta}_k}(\zeta_0) = 0$, we have at $[(\zeta_0,w_0)] \in \partial E$
\begin{equation*}
\frac{\partial^2}{\partial \zeta_j \partial \bar \zeta_k } \left(- \log (1-|s|^2 )\right) = s_{\zeta_j} \cdot \overline{s}_{\bar \zeta_k}  + \overline {s}_{\zeta_j} \cdot s_{\bar \zeta_k}
\end{equation*}
and
\begin{equation*}
\begin{aligned}
\frac{\partial}{\partial \zeta_j\partial \bar \zeta_k } \bigg( \log (1- s \cdot \overline w ) + \log (1- \overline s \cdot w) \bigg)
&= -\frac{\partial}{\partial \zeta_j} \bigg(\frac{ s_{\overline \zeta_k} \cdot \overline w }{1- s \cdot\overline w } + \frac{\overline s_{\overline \zeta_k} \cdot w}{1-  \overline s \cdot w} \bigg) \\
&= {( s_{\zeta_j} \cdot \overline w_0) ( s_{\overline \zeta_k} \cdot \overline w_0)}{} + {(w_0 \cdot \overline s_{\zeta_j} ) (w_0 \cdot \overline s _{\overline \zeta_k})}{} .
\end{aligned}
\end{equation*}
Then, summing up, we obtain
\begin{equation*}
\begin{aligned}
\frac{1}{2}\frac{\partial^2 (-\log |g(\tau, \nu)|^2)}{\partial \zeta_j \partial \overline \zeta_k} 
&=  (s_{\zeta_j} \cdot \overline{s}_{\bar \zeta_k}  + \overline {s}_{\zeta_j} \cdot s_{\bar \zeta_k}) 
+(s_{\zeta_j} \cdot \overline w_0) ( s_{\overline \zeta_k} \cdot \overline w_0)
+  ( \overline s_{\zeta_j} \cdot w_0) ( \overline s _{\overline \zeta_k} \cdot w_0) 
. 
\end{aligned}
\end{equation*}

Therefore, by the Cauchy-Schwarz inequality, we estimate the curvature tensor as
\begin{equation*}
\begin{aligned}
\frac{1}{2}\Theta_{g_{\mathcal N}} (X,\overline{X}) &=-\frac{1}{2}\sum_{j,k}\frac{\partial^2 \log |g(\tau, \nu)|^2}{\partial \zeta_j \partial \overline \zeta_k} X_j \overline X_k\\
&= |\partial s(X)|^2 + |\bar \partial s(\overline X)|^2 + 2 \operatorname{Re} ( \partial s(X) \cdot \overline w_0)( \overline\partial s(\overline X) \cdot \overline w_0)\\
&\geq |\partial s(X)|^2 + |\overline\partial s(\overline X)|^2 - 2 |\partial s(X)\cdot \overline w_0| | \overline\partial s(\overline X) \cdot \overline w_0|\\
&\geq |\partial s(X)|^2 + |\overline\partial s(\overline X)|^2 - 2 |\partial s(X)| | \overline\partial s(\overline X)|\\
&= (|\partial s(X)| - |\overline\partial s(\overline X)|)^2 \geq 0
\end{aligned}
\end{equation*}
where $X = \sum_j X_j \frac{\partial}{\partial \zeta_j} \in T^{1,0}_{\mathcal{F}, [(\zeta_0, w_0)]}$ is identified with $\pi_* X \in T^{1,0}_{\Sigma, [\zeta_0]}$. 
This shows that the leafwise curvature of $g_\mathcal{N}$ is semi-positive. 

Finally, we examine the degeneracy locus of curvature. Assume that the curvature vanishes at  $[(\zeta_0,w_0)] \in \partial E$ for non-zero $X \in T^{1,0}_{\mathcal{F}, [(\zeta_0, w_0)]}$.
Since the equality holds in the Cauchy--Schwarz inequality used above and $|\partial s(X)| = |\overline\partial s(\overline X)|$, $\partial s(X) = e^{i\theta} \overline\partial s(\overline{X})$ holds at $[\zeta_0] \in \Sigma$ for some $e^{i\theta} \in U(1)$.
This implies that $ds(X - e^{i\theta}\overline{X}) = 0$, as $X$ is of $(1,0)$ form, and $ds_{\zeta_0} \colon T_{\zeta_0}\widetilde{\Sigma} \to T_{s(\zeta_0)}\mathbb{B}^n$ has non-trivial kernel.
\end{proof}

\begin{remark} \label{rem:non-positivity}
The degeneracy locus $C$ of the metric constructed in Proposition \ref{prop:semi-positive} coincides with the entire $\Sigma$ if $\dim_{\mathbb{C}} \Sigma > n$, where $n$ is the dimension of fiber $\mathbb{B}^n$. 
We conjecture that $\mathcal{N}$ would never admit a smooth hermitian metric of positive curvature along the horizontal foliation $\mathcal{F}$ when $\Sigma$ is a compact complex manifold of dimension $> n$ and $n \geq 2$. 
Note that Brinkschulte's theorem \cite{Brinkschulte} implies that $\mathcal{N}$ does not admit a smooth hermitian metric of positive curvature along $\mathcal{F}$ when $\Sigma$ is a compact complex manifold of dimension $> 1$ and $n =1$ since in this case $\partial E$ is a Levi-flat real hypersurface and $\mathcal{N}$ agrees with the normal bundle of the Levi foliation.
\end{remark}

\section{Proof of Theorems}

\begin{proof}[Proof of Theorem \ref{thm:main}]
From Proposition \ref{prop:semi-positive}, we have a smooth hermitian metric $g_{\mathcal N}$ on $\mathcal N$ of semi-positive curvature along $\mathcal F$. The curvature of $g_{\mathcal N}$ degenerates exactly on $\pi_{\partial E}^{-1}(C)$, where $C$ is the set of points where rank of $ds$ is less than two.
From the assumption, we have a point $p$ of maximal rank of $s$, hence, $C \neq \Sigma$. There exists a relatively compact open neighborhood $p \in W \Subset \Sigma \setminus C$. Take smaller open neighborhoods $p \in U \Subset V \Subset W$. 
Since $\Sigma \setminus \overline{U}$ is an open Riemann surface, there exists a smooth strictly subharmonic function $\psi$ on $\Sigma \setminus \overline{U}$ thanks to the Behnke--Stein theorem. By a smooth cut-off function, we extend $\psi|_{\Sigma \setminus V}$ to a smooth function $\widehat{\psi}$ on $\Sigma$. 
We may choose $0 < \varepsilon \ll 1$ so that $\sqrt{-1}\Theta_{g_{\mathcal{N}}}  + \varepsilon \pi_{\partial E}^*(\sqrt{-1}\partial\overline{\partial} \widehat \psi) > 0$ as leafwise $(1,1)$-form on $\pi^{-1}_{\partial E}(W)$ since $\sqrt{-1}\Theta_{g_{\mathcal{N}}} > 0$ on $\pi^{-1}_{\partial E}(W)$.
Then ${\widehat g}_\mathcal{N} := e^{-(\widehat{\psi} \circ \pi_{\partial E})} g_\mathcal{N}$ gives a smooth hermitian metric of positive leafwise curvature over $\partial E$. 
\end{proof}

\begin{proof}[Proof of Theorem \ref{cor:1-n-covex}]
We adapt the construction of exhaustion functions on the complement of foliated real hypersurfaces, which originates in the work of Brunella \cite{Brunella} and was used in the study of the Diederich--Forn{\ae}ss index as in \cite{Adachi-Yum} and \cite{Dallara-Mongodi}.

First we shall relate the leafwise postively curved metric ${\widehat g}_\mathcal{N}$ of $\mathcal{N}$ with a pseudo-hermitian structure of $\partial E$.
Consider the global smooth frame $\nu$ of $\mathcal{E}^\perp$ used in the proof of Propositon \ref{prop:semi-positive}.
We write the coordinates of $V^2_\zeta$ as $t_j = u_j + \sqrt{-1}v_j$.
The imaginary part of $\nu$,
$$\operatorname{Im} \nu =   \sum_{\ell=1}^n \operatorname{Im} \left(t_\ell \frac{\partial}{\partial t_\ell}\right) = \frac{1}{2} \sum_{j=1}^n \left( v_j \frac{\partial}{\partial u_j} - u_j \frac{\partial}{\partial v_j}\right),$$ 
is a non-vanishing real smooth vector field on $\partial E$ that is transverse to the maximal complex distribution $H(\partial E) := \operatorname{Re} T^{1,0}_{\partial E}$. We regard $\operatorname{Im} \nu$ as a global smooth frame of the quotient real line bundle $T(\partial E)/H(\partial E)$.
Combining with the positive smooth function $ |\nu|_{\widehat{g}_\mathcal{N}}$ on $\partial E$, we define a pseudo-hermitian structure $\theta$ on $\partial E$ by 
$$
\theta_p(v_p) = |\nu_p|_{\widehat{g}_\mathcal{N}} \frac{v_p}{\operatorname{Im} \nu_p} \quad \text{for $v_p \in T_p(\partial E)/H_p(\partial E)$, $p \in M$}.
$$

Next we compute the D'Angelo $(1,0)$-form $\alpha_{1,0}$ with respect to ${\theta}$. 
Consider the local leafwise holomorphic vector field $\tau$ in the proof of Proposition \ref{prop:semi-positive}, and define a real smooth vector field $T$ on $\pi^{-1}_{\partial E}(D)$ by
$
T := \operatorname{Im} \tau / |\tau|_{\widehat{g}_\mathcal{N}}.
$
This local vector field is $\theta$-normalized because
$$
\theta(T) = \frac{ |\nu|_{\widehat{g}_\mathcal{N}} }{|\tau|_{\widehat{g}_\mathcal{N}}} \frac{\operatorname{Im}\tau}{ \operatorname{Im}\nu}
= 1,
$$
where the quotient is taken in $T(\partial E)/H(\partial E)$, follows from $\tau/|\tau|_{\widehat{g}_\mathcal{N}} = \nu/|\nu|_{\widehat{g}_\mathcal{N}}$ in $\mathcal{N}$.
Since $\operatorname{Im} \tau$ is independent of $\zeta$,
$$
\alpha_{1,0}(\frac{\partial}{\partial \zeta}) = \theta([T, \frac{\partial}{\partial \zeta}])
= - |\tau|_{\widehat{g}_\mathcal{N}} \frac{\partial}{\partial \zeta}\frac{1}{|\tau|_{\widehat{g}_\mathcal{N}}}
= \frac{\partial}{\partial \zeta} \log |\tau|_{\widehat{g}_\mathcal{N}}.
$$
It follows that
$$
\overline{\partial} \alpha_{1,0} =  - \frac{\partial^2}{\partial \zeta \partial \overline{\zeta}} \log |\tau|_{\widehat{g}_\mathcal{N}} d\zeta \wedge d\overline{\zeta} =
\frac{1}{2} \Theta_{\widehat{g}_\mathcal{N}} > 0.
$$
From Proposition \ref{prop:adachi-yum}, this positivity implies that $E$ is 1-convex. 
Also, from Proposition \ref{prop:q-convexity}, this implies that $E'$ is $n$-convex. 

Assume, for the sake of contradiction, that $E'$ is $(n-1)$-convex for $n \geq 2$. Then the fiber $\mathbb{CP}^n \setminus \overline{\mathbb{B}^n}$ is also $(n-1)$-convex and by the finiteness theorem of Andreotti and Grauert \cite{AG}, the Dolbeault cohomology group $H^{n,n-1}(\mathbb{CP}^n \setminus \overline{\mathbb{B}^n})$ is finite dimensional. However, by Serre's duality, the compactly supported Dolbeault cohomology group $H^{0,1}_c(\mathbb{CP}^n \setminus \overline{\mathbb{B}^n})$ is also finite dimensional. 
This enables us to find a non-constant holomorphic function defined on a neighborhood of $\overline{\mathbb{B}^n}$ that holomorphically extends on $\mathbb{CP}^n$;
Consider $f_j = z_1^j$ for $j = 1, 2, \dots$. 
Let $\chi \colon \mathbb{C}^n \to [0,1]$ be a smooth cut-off function such that $\chi(z) = 1$ if $|z| < 2$ and $\chi(z) = 0$ if $|z| > 3$. 
The $(0,1)$-forms $\overline{\partial}(\chi f_j)$ can be seen as compactly supported forms on $\mathbb{CP}^n \setminus \overline{\mathbb{B}^n}$, which are linearly independent.
Since $H_c^{0,1}(\mathbb{CP}^n \setminus \overline{\mathbb{B}^n})$ is assumed to be finite dimensional, there exists some $j$ such that $\overline\partial g = \overline{\partial}(\chi f_j)$ has a smooth solution $g$ compactly supported in $\mathbb{CP}^n \setminus \overline{\mathbb{B}^n}$. Then $\chi f_j - g$ is a non-constant holomorphic function on $\mathbb{CP}^n$ extending $f_j$.
This contradicts the compactness of $\mathbb{CP}^n$.
\end{proof}

\section*{Acknowledgments}
The authors dedicate this paper to Professor Kang-Tae Kim
in honor of his outstanding contributions to several complex variables and complex
geometry. We appreciate Masakazu Takakura for pointing out that a locally trivial holomorphic $\mathbb{B}^n$-bundle over a compact complex manifold of dimension $k$ cannot be Stein if $k > n$.
We also thank Takeo Ohsawa for pointing out an inaccuracy in a draft of this paper.
The first author was partially supported by JSPS KAKENHI Grant Numbers JP24K06776 and JP25K21994. The second author was
partially supported by the National Research Foundation of Korea (NRF) grant funded by
the Korean government (MSIT) (no. RS-2024-00339854). {The third author was supported by the Basic Science Research Program
through the National Research Foundation of Korea (NRF) funded by the Ministry of
Education (RS-2025-00561084).}

\end{document}